\documentclass[11pt]{article}

\usepackage{tikz}  
\usepackage{amsmath}
\usepackage{amsfonts}
\usepackage{amssymb}
\usepackage{amsthm}

\newtheorem{theorem}{Theorem}
\newtheorem{proposition}{Proposition}
\newtheorem{definition}{Definition}
\newtheorem{example}{Example}
\newtheorem{remark}{Remark}
\newtheorem{corollary}{Corollary}

\title{Combinatorics of NC-Probability Spaces with Independent Constants} 

\author{%
   Carlos Diaz-Aguilera\footnote{e-mail: carlos.edo.diaza@gmail.com}
    \and 
   Tulio Gaxiola\footnote{Universidad Autónoma de Sinaloa, e-mail:  tuliogax@uas.edu.mx}
  \and 
  Jorge Santos\footnote{UNAM, Instituto de Matemáticas, e-mail:  el.santos.el@gmail.com
    }
    \and Carlos Vargas\footnote{e-mail: obieta@gmail.com}
    }

\newcommand{\partition}[3]{
  \begin{tikzpicture}[x=0.22cm,y=0.3cm,line cap=round,line width=0.25mm]
    \draw (0,0) -- (0,1);
    \draw (1,0) -- (1,1);
    \draw (2,0) -- (2,1);
    \draw (3,0) -- (3,1);

    \ifnum#1=0%
      \draw (0.5,1.4) node[centered] {$\bullet$};
    \else
      \draw (0,1) -- (1,1);
      \draw (0.5,1.4) node[centered] {$\times$};
    \fi
    
    \ifnum#2=0%
      \draw (1.5,1.4) node[centered] {$\bullet$};
    \else
      \draw (1,1) -- (2,1);
      \draw (1.5,1.4) node[centered] {$\times$};
    \fi

    \ifnum#3=0%
      \draw (2.5,1.4) node[centered] {$\bullet$};
    \else
      \draw (2,1) -- (3,1);
      \draw (2.5,1.4) node[centered] {$\times$};
    \fi
  \end{tikzpicture}
}

\newcommand{\partitionb}[3]{
  \begin{tikzpicture}[x=0.22cm,y=0.3cm,line cap=round,line width=0.25mm]
    \draw (0,0) -- (0,1);
    \draw (1,0) -- (1,1);
    \draw (2,0) -- (2,1);

    \ifnum#1=0%
      \draw (0.5,1.4) node[centered] {$\bullet$};
    \else
      \draw (0,1) -- (1,1);
      \draw (0.5,1.4) node[centered] {$\times$};
    \fi
    
    \ifnum#2=0%
      \draw (1.5,1.4) node[centered] {$\bullet$};
    \else
      \draw (1,1) -- (2,1);
      \draw (1.5,1.4) node[centered] {$\times$};
    \fi

    \ifnum#3=0%
      \draw (2.5,1.4) node[centered] {$\bullet$};
    \else
      \draw (2,1) -- (3,1);
      \draw (2.5,1.4) node[centered] {$\times$};
    \fi
  \end{tikzpicture}
}

\newcommand{\partitionc}{
  \begin{tikzpicture}[x=0.22cm,y=0.3cm,line cap=round,line width=0.25mm]
    \draw (0,0) -- (0,1);
    \draw (1,0) -- (1,1);
    \draw (2,0) -- (2,1);
    \draw (0,1) -- (2,1);
  \end{tikzpicture}
}

\newcommand{\partitiond}{
  \begin{tikzpicture}[x=0.22cm,y=0.3cm,line cap=round,line width=0.25mm]
    \draw (0,0) -- (0,1);
    \draw (1,0) -- (1,0.6);
    \draw (2,0) -- (2,1);
    \draw (0,1) -- (2,1);

    \draw (0.5,1.4) node[centered] {$\bullet$};
    \draw (1.5,1.4) node[centered] {$\bullet$};
    \draw (2.5,1.4) node[centered] {$\times$};
  \end{tikzpicture}
}

\begin{document}
\maketitle
\begin{abstract}
    
Unlike classical and free independence, the boolean and monotone notions of independence lack of the property of independent constants. 

In the scalar case, this leads to restrictions for the central limit theorems, as observed by F. Oravecz.

We characterize the property of independent constants from a combinatorial point of view, based on cumulants and set partitions. This characterization also holds for the operator-valued extension.

Our considerations lead rather directly to very mild variations of boolean and monotone cumulants, where constants are now independent. These alternative probability theories are closely related to the usual notions. Hence, an important part of the boolean/monotone probability theories can be imported directly. 

We describe some standard combinatorial aspects of these variations (and their cyclic versions), such as their M\"obius functions, which feature well-known combinatorial integer sequences. 

The new notions with independent constants seem also more strongly related to the operator-valued extension of c-free probability.
\end{abstract}

\section{Introduction}

Non-Commutative or quantum probability goes back to the 70's, with the works of Cushen, Hudson, Giri and von Waldenfels (\cite{CH71, Hud73, GW78, Wal78}). The main idea is to understand the notion of stochastic independence as an algebraic relation, and then to consider new algebraic relations (in particular between non-commuting algebras of operators) as new notions of independence, hopefully leading to robust probability theories, with their own limit theorems, Gaussian/Poisson distributions, models and applications. 

In the following couple of decades, D. V. Voiculescu introduced and developed free probability theory (\cite{Voi83}), one of the most prominent branches of non-commutative probability.

R. Speicher investigated free probability from a combinatorial perspective, using set partitions and cumulants, \cite{Spe94} and included a new (though somewhat simpler), boolean probability theory \cite{SW97} (see also c-free probability \cite{BS91, BLS96}). 

A. Ben-Ghorbal and M. Sch\"urmann provided a categorical classification for non-commutative notions of stochastic independence of operator algebras \cite{BS02} (based on the earlier combinatorial classification \cite{Spe97}). N. Muraki extended the categorical classification to non-symmetric notions of independence \cite{Mur03} (which included the monotone independence \cite{Mur01}, see also U. Franz \cite{Fra01} and Hasebe-Saigo \cite{HS11}).

Voiculescu established concrete applications for free probability in his seminal work \cite{Voi91}, providing a more conceptual approach to the understanding of limiting eigenvalue distributions from random matrix models, started by E. Wigner's \cite{Wig58}, extended by V. Marcenko and L. Pastur \cite{MP67}, and more systematically by V. Girko. 

The application of free probability to such models, and more, has been very successfully addressed by multiple authors in the last century. Many of these applications are based on the crucial concept of operator-valued free independence/probability, introduced by Voiculescu in \cite{Voi95} (see also \cite{Spe98}).

When it comes to boolean independence, for non-trivial cases, the constants with respect to a conditional expectation $\mathbb F:\mathcal A\to \mathcal B$ are \emph{not independent} from the rest of the algebra, as it occurs in the free or classical situation. 

This leads to some immediate problems. For example, the boolean and monotone Gaussian distributions are somehow less universal, as the Central Limit Theorems require \emph{centered} independent non-commutative random variables. At the operator-valued level, one might say that it is objectively hard to present a robust model for $\mathcal B$-valued boolean independent variables (although both monotone and boolean probability theories are quite useful for describing free probabilistic objects, and a \emph{cyclic} monotone independence has found practical applications to random matrix theory \cite{CHS18}). 

\subsection{Main Contributions and Organization}

The conditions for independent constants can be conveniently characterized in terms of cumulants and set partitions (or more generally, on a property for weights on families of set partitions, see Section 3). 

The main concept that we discuss is the following property (for definitions about non-commutative probability, partitions, cumulants, and weights, see Section 2).

\begin{definition}
Consider the posets of set partitions $\mathcal P(n)$ of the set $\{1,2,3,\dots n\}$, $n\geq 1$ (with partial order of reverse refinement of blocks). Denote by $0_n,1_n\in \mathcal P(n)$, respectively, the minimum and maximum element of $\mathcal P (n)$. 

A subfamily $\mathcal L=\bigcup_n \mathcal L(n) \subseteq \mathcal P$ of set partitions, with $L(n)\subseteq \mathcal P(n)$ has the \emph{Singleton-Inductive} property (SI) iff for any $n\geq 0$ and $r\leq n+1$, the sub-poset $\{\pi \in \mathcal L(n+1): \{r\}\in \pi\}\subseteq \mathcal L(n+1)$ is isomorphic to $\mathcal L(n)$.
\end{definition}
All partitions ($\mathcal L=\mathcal P$, corresponding to classical independence) and non-crossing partitions ($\mathcal L=\mathcal {NC}$, to free independence) are SI, with the isomorphisms given, respectively, by $\Psi^n_r: \mathcal P(n)\to \mathcal P(n+1)$, which simply adds the singleton $\{r\}$ to each partition $\pi\in \mathcal P(n)$ (shifting the rest of the partition, see Section 3), and its restriction to $\mathcal {NC}$. 

Interval partitions ($\mathcal L=\mathcal I$, boolean independence) are not SI, and neither are the monotone weights on partitions.

Our main theorem states that SI families of partitions (or weights) produce cumulants respecting the independence of constant random variables.
\begin{theorem}\label{mainthm}
Let $(\mathcal A,\mathcal B, \mathbb F)$ be an OVPS and consider cumulants $(c_n)_{n\geq 1}$ w.r.t. an SI subfamily of set partitions $\mathcal L$. 

Then for all $n\geq 2$ and any $a=(a_1,a_2,\dots,a_n)\in \mathcal A^n$, the cumulant $c_n(a)=0$, whenever there is any $a_i\in \mathcal B$, $i\leq n$. In other words, the algebra of constants $\mathcal B$ is independent from the whole algebra $\mathcal A$.
\end{theorem}

Proof. Section 3.

We explore examples of SI families in Section 3. In particular, we slightly modify the boolean interval partitions, the monotone weights on non-crossing partitions, and their cyclic versions, to make them SI.

We comment on the modified boolean probability theory (Section 4). We begin by understanding its intersection with the usual boolean probability. At the scalar-valued level, we recover Oravecz's Fermi-boolean probability theory. We also briefly discuss connections with c-free probability.

In Section 5 we describe the new lattices of interval partitions and their M\"obius functions (which are useful for computing cumulant-to-cumulant formulas). Our results feature standard combinatorial recursions and sequences.

We thank O. Arizmendi, A. Beshenov and H. Yoshida for useful comments and remarks.

\section{Preliminaries}

\subsection{Non-Commutative Probability Spaces}

\begin{definition}
An Operator-Valued (Non-Commutative) Probability Space (or, simply, $\mathcal B$-valued probability space), is a triplet $(\mathcal{A},\mathcal{B},\mathbb{F})$ (often denoted just by $(\mathcal{A},\mathbb{F})$), where:

$\mathcal{A}$ is a $*$-algebra with multiplicative unit $1_{\mathcal A}$,

$\mathcal{B}\subseteq \mathcal{A}$ is a sub-algebra of $\mathcal{A}$ which may not contain the unit, and 

$\mathbb{F}\colon\mathcal{A}\rightarrow\mathcal{B}$ is a conditional expectation. This means that $\mathbb{F}$ is a $\mathcal{B}$-linear map satisfying, for all $a,\mathcal{A}$ and $b,b'\in \mathcal{B}$, $$\mathbb{F}(bab')=b\mathbb{F}(a)b', \quad \mathbb F(b)=b.$$
We say that two operator-valued probability spaces $(\mathcal{A},\mathcal{B}_1,\mathbb{F}_1)$, $(\mathcal{A},\mathcal{B}_2,\mathbb{F}_2) $ are \emph{compatible} if $\mathcal{B}_1\subseteq \mathcal{B}_2$ and $\mathbb{F}_1\circ\mathbb{F}_2=\mathbb{F}_1$.
\end{definition}

The elements $a_1,a_2,\dots a_n \in \mathcal A$ are called non-commutative random variables. We denote by $\langle a_1,a_2,\dots, a_n \rangle\subseteq \mathcal A$ the non-unital algebra that they generate.

The usual notion of stochastic independence in probability theory, is a relation between unital algebras of commuting random variables, which can be established in terms of the factorization of mixed moments, whenever these determine the joint distributions. 

For example, two real-valued random variables $X,Y$ with compact support commute and their mixed moments determine their joint distributions. Thus the independence of $X$ and $Y$ can be characterized by the factorization of mixed moments:
$$\mathbb E(X^nY^m)=\mathbb E(X^n)\mathbb E(Y^m),\quad \text{for all } m,n\geq 0.$$
If (possibly non-commuting) elements in operator algebras are regarded as non-commutative random variables with respect to a functional or a conditional expectation, notions of independence may be defined, including a generalization of classical independence, called \emph{tensor independence}, but also some new alternatives.  
\begin{definition}
Let $(\mathcal{A},\mathcal{B},\mathbb{F})$ be 
an OVPS and $\mathcal{A}_1,\mathcal{A}_2,\dots,\mathcal{A}_n\subseteq \mathcal{A}$ sub-algebras (not necessarily unital) such that $\mathcal{B}\subseteq \mathcal{A}_i$ for all $i\leq n$. 
We say that:
\begin{enumerate}
    \item The algebras $\mathcal{A}_1,\mathcal{A}_2,\dots,\mathcal{A}_n$ are $\mathcal{B}$-\emph{tensor independent} iff elements from \emph{different} algebras commute and  $$\mathbb{F}(a_1a_2\dots a_n)=\mathbb{F}(a_1)\mathbb{F}(a_2)\cdots \mathbb{F}( a_n),$$ holds for any tuple with $a_i\in\langle \mathcal{A}_i, 1_{\mathcal A}\rangle$.
\item Let $\dot{a}=a-\mathbb{F}(a)$ for all $a\in \mathcal A$. The algebras $\mathcal{A}_1,\mathcal{A}_2,\dots,\mathcal{A}_n$ are $\mathcal{B}$-\emph{free} iff  $$\mathbb{F}(\dot{a}_1\dot{a}_2\dots \dot{a}_k)=0,$$ for all $k\geq 0$ and all tuples $a_i\in\mathcal{A}_{j(i)} ,i\leq k$, such that $j(1)\neq j(2)\neq \cdots \neq j(k)$.  
    
    \item  The algebras $\mathcal{A}_1,\mathcal{A}_2,\dots,\mathcal{A}_n$ are $\mathcal{B}$-\emph{boolean independent} iff $$\mathbb{F}(a_1a_2\cdots a_m)=\mathbb{F}(a_1)\mathbb{F}(a_2)\cdots \mathbb{F}(a_m),$$ for any $m\geq 1$ and any $a_i\in \mathcal{A}_{j(i)}$, $i\leq n$, such that $j(1)\neq j(2)\neq \cdots \neq j(k)$.
    \end{enumerate}
    \end{definition}

\begin{remark}
Although each factorization seems to apply for a specific form of mixed moment, one should note that they actually allow us to express any mixed moment in terms of moments restricted to the independent algebras. For example, for boolean independent $a_1,a_2,a_3$ we get that $$\mathbb F(a_1a_1a_2a_2a_1a_1a_3a_3a_3a_2a_2a_3a_3)=\mathbb F(a_1^2)\mathbb F(a_2^2)\mathbb F(a_1^2)\mathbb F(a_3^3)\mathbb F(a_2^2)\mathbb F(a_3^2),$$ whereas for tensor independence we get $\mathbb F(a_1^4)\mathbb F(a_2^4)\mathbb F(a_3^5)$. For the free case the expression needs to be solved inductively.

Tensor and classical independence extend to the unital algebras (i.e. $\mathcal A_1$,$\mathcal A_2$ are independent iff  $\langle \mathcal A_1,1_{\mathcal A}\rangle$,$\langle \mathcal A_2,1_{\mathcal A}\rangle$ are independent).
\end{remark}

\begin{remark}
Muraki's Monotone independence \cite{Mur01} is a non-symmetric, non-unital notion of independence (associative, see \cite{Fra01}). It is simpler to describe it through examples:

If $(a_1,a_2,a_3)$ are monotone independent (in that order), then we have, for instance,
$$\mathbb F(a_1a_2a_1a_2a_1a_2a_1a_2)=\mathbb F(a_1\mathbb F(a_2)a_1\mathbb F(a_2)a_1\mathbb F(a_2)a_1\mathbb F(a_2)).$$
We also have, for example, $$\mathbb F(a_1a_1a_2a_2a_3a_3a_3a_2a_2a_1a_1a_2a_2a_3a_3)=\mathbb F(a_1^2\mathbb F(a_2^2\mathbb F(a_3^3)a_2^2)a_1^2\mathbb F(a_2^2)\mathbb F(a_3^2)).$$

In order to factorize a general mixed moment, one first evaluates on the intervals of the last algebra (in the last example, the two intervals of the variable $a_3$). After evaluating $\mathbb F$ in those intervals, they become elements of $\mathcal B$ which are attached to their neighboring elements from the remaining algebras. 

Then one evaluates on the new interval partitions involving elements of the next algebra ($\langle a_2,\mathcal B \rangle$ for the example). And so on, until the last evaluation of $\mathbb F$ on the (necessarily unique) interval of elements from the remaining algebra ($\langle a_1, \mathcal B \rangle$ in this case). 
\end{remark}

\subsection{Set-Partitions and Cumulants}
\begin{definition}
Let $n\geq 1$:
\begin{enumerate}
    \item A \emph{partition} is a decomposition $\pi=\{V_1,V_2,\dots,V_k\}$ of the set $[n]:=\{1,2,\dots, n\}$ into non-empty, mutually disjoint subsets $V_1,V_2,\dots, V_k\subseteq [n]$ called \emph{blocks}.
    \item We say that two blocks $V\neq W\in \pi$  have a \emph{crossing}, if there is a tuple $1\leq a<b<c<d\leq n$, with $a,c\in V$ and $b,d\in W$. If $\pi$ has no pair of such blocks, we say that $\pi$ \emph{non-crossing}. 
    \item A block $V\in \pi$ in a non-crossing partition is an \emph{inner block} if there exists a different block $W\in \pi$ and $i,j\in W$ with $i<k<j$ for all $k\in V$. A block that is not an inner block is called \emph{outer block}.
    \item A partition $\pi$ is an \emph{interval} partition if it only contains outer blocks, (and thus the blocks are intervals of consecutive numbers).  
\end{enumerate}
The sets of partitions, non-crossing partitions and interval partitions of $[n]$ will be denoted by $\mathcal{P}(n)$, $\mathcal{NC}(n)$ and $\mathcal{I}(n)$ respectively. We also write $\mathcal P:=\bigcup_{n\geq 1} \mathcal P(n)$ and similarly for $\mathcal{NC}$ and $\mathcal{I}$. 
\end{definition}
Set-partitions are partially ordered sets (with the order of reverse refinement of blocks). Now we recall the definitions for multiplicative families of maps and cumulants (see \cite{Spe98}).
\begin{definition}\label{multext}
Let $\mathcal{A}$ an algebra and $\mathcal{B}\subset \mathcal{A}$ a subalgebra.

1. Suppose that the commutator $[\mathcal{A},\mathcal{B}]$ is trivial. Then we define the multiplicative extension to the set of all partitions $(f_{\pi})_{\pi\in \mathcal P}$ of a collection of multilinear maps $(f_n)_{n\geq 1}$, $f_n:\mathcal A^n\to \mathcal B$ as follows: For $\pi=\{V_1, V_2, \dots , V_k\} \in \mathcal P (n)$, $f_{\pi}:\mathcal A^n\to \mathcal B$ is defined as:
$$f_{\pi}(a_1,a_2,\dots, a_n):=\prod_{i\leq k}f_{|V_i|}(a_{j_1},a_{j_2},\dots, j_r),$$
where $V_i=(j_1,j_2, \dots, j_r)$.

2. In the general case, when $[\mathcal{A},\mathcal{B}]$ is non-trivial, we may only define the multiplicative extension of a family of maps for non-crossing partitions $(f_{\pi})_{\pi\in \mathcal NC}$. For $\pi=\{V_1, V_2, \dots , V_k\} \in \mathcal {NC} (n)$, consider an interval block $V_i=(j,j+1, \dots, j+r-1)$ from $\pi$, and define $f_{\pi}:\mathcal A^n\to \mathcal B$ inductively as: $$f_{\pi}(a_1,a_2,\dots , a_n)=f_{\pi\setminus V_i}(a_1,a_2,\dots,f_{r}(a_j, a_{j+1},\dots, a_{j+r-1})a_{j+r},\dots a_n).$$
\end{definition}

\begin{remark}
1. One should be aware that any non-crossing partition contains at least one interval block. 

2. In principle one could also define
$$f_{\pi}(a_1,a_2,\dots , a_n)=f_{\pi\setminus V_i}(a_1,a_2,\dots a_{j-1}f_{r}(a_j, a_{j+1},\dots, a_{j+r-1}),a_{j+r},\dots a_n).$$
The maps for which the two expressions coincide are called $\mathcal B$-balanced. For balanced maps, the previous definition does not depend on the order in which interval blocks are chosen.

The properties of conditional expectations make $\mathbb F_n:(a_1,a_2,\dots, a_n)\mapsto \mathbb F(a_1a_2\cdots a_n)$, $n\geq 1$ a $\mathcal B$-balanced family of maps.
\end{remark}

\begin{definition}
Classical $\mathcal{B}$-Cumulants. 
Let $(\mathcal{A},\mathcal{B},\mathbb{F})$, be an OVPS with $[\mathcal A,\mathcal B]=0$. The classical cumulants are the $\mathcal B$ multi-linear maps $c_n\colon \mathcal{A}^n\rightarrow \mathbb{C}
$ which are defined inductively by the moment-cumulant formulas $$\mathbb F(a_1a_2\dots a_n)= \sum_{\pi\in \mathcal{P}(n)} c_\pi(a_1,a_2,\dots, a_n),$$ where $c_\pi(a_1,a_2,\dots, a_n)$ is defined as in the first paragraph of Definition \ref{multext}.
\end{definition}

\begin{definition}
Let $(\mathcal{A},\mathcal{B},\mathbb{F})$ a $\mathcal{B}$-valued space, the free cumulants are a family $(r_n)_{n\geq 1}$ of $\mathcal B$-multi-linear maps, $r_n\colon \mathcal{A}^n\rightarrow\mathcal{B}$ defined inductively by the equations $$\mathbb{F}(a_1a_2\dots a_n)= \sum_{\pi\in \mathcal{NC}(n)} r_\pi(a_1,a_2,\dots, a_n),$$ with $r_\pi(a_1,a_2,\dots, a_n)$ defined as in the second paragraph of Definition \ref{multext}.

Similarly, the boolean cumulants are the family of maps $(b_n)_{n\geq 1}$, $b_n\colon \mathcal{A}^n\rightarrow\mathcal{B}$ defined inductively by the equations $$\mathbb{F}(a_1a_2\dots a_n)= \sum_{\pi\in \mathcal{I}(n)} b_\pi(a_1,a_2,\dots, a_n).$$
\end{definition}

\begin{theorem}[\cite{Spe94, SW97, Spe98}]
Cumulants characterize independence.
The algebras $\{\mathcal{A}_i\}_{i\leq n}$ are independent (in the boolean, tensor or free sense) if and only if all the corresponding mixed cumulants vanish.
\end{theorem}
To make our main theorems more general, so that they are also applicable to non-symmetric notions of independence let us give the following extensions to weights
\begin{definition}
A family of weights on $\mathcal P$ is simply a function $\omega=(\omega_n)_n:\mathcal P\to \mathbb C$, where $\omega_n:\mathcal P(n)\to \mathbb C$. When $n$ is clear from the context, we may write $\omega(\pi)$ instead of $\omega_n(\pi)$.

We may generalize moment-cumulant formulas for weights on $\mathcal P$, as
$$\mathbb{F}(a_1a_2\dots a_n)= \sum_{\pi\in \mathcal{NC}(n)} \omega(\pi) c^{\omega}_{\pi}(a_1,a_2,\dots, a_n),$$
(using the appropriate definition for multiplicative families from \ref{multext}).

For cumulants to be well-defined recursively from moments, we just need $\omega_n(1_n)=\lambda_n\neq 0$, for all $n\geq 1$. We call this kind of weights invertible.

For our purposes, we will only consider weigths with the normalization condition $\omega_1(\{\{1\}\})=1$, so that the first cumulant equals the first moment. 

Often, one requests a monic condition: $\omega_n(1_n)=1$, for all $n\geq 1$ (see \cite{Leh02}), so that the $n$-th moment cumulant formula is monic on the $n$-th cumulant, and hence, in particular, cumulants can always be solved inductively. 
\end{definition}

\section{SI Families of Partitions}
\subsection{Singleton Inductive Weights on Partitions}
Let us first give a more general definition of the property SI, for families \emph{of weights} on partitions (and not only for families of subsets, where the weight is just the subset's indicator function). This will allow us to address the monotone situation as well.
\begin{definition}
For any $n\geq 1$ and any $1\leq r\leq n+1$, we define the poset-homomorphisms $$\Psi^{(n)}_r:\mathcal P(n)\to [0_{n+1},\{\{1,2,\dots,r-1,r+1,\dots ,n+1\},\{r\}\}]\subset \mathcal P(n+1), \quad r\leq n,$$ and $$\Psi^{(n)}_{n+1}:\mathcal P(n)\to [0_{n+1},\{\{1,2,\dots n\},\{n+1\}\}]\subset \mathcal P(n+1),$$ which simply inserts an additional singleton at the position $\{r\}$ on the partition $\pi$, shifting the rest of the partition. 
\end{definition}
For example: for $\pi=\{\{1,3\},\{2,4\}\}$, we have that $\Psi^{(4)}_1(\pi)=\{\{1\},\{2,4\},\{3,5\}\}$, $\Psi^{(4)}_3(\pi)=\{\{1,4\},\{2,5\},\{3\}\}$, and $\Psi^{(4)}_5(\pi)=\{\{1,3\},\{2,4\},\{5\}\}$. 
\begin{definition}
A family of weights $\omega=\{\omega_{i}\}_{i\geq 1}$ on set partitions $\mathcal{P}=\bigcup_{i\geq 1} \mathcal{P}(i)$ is called \emph{Singleton-Inductive}, iff $\omega_1(\{1\})=1$, and all the maps $\Psi^{(n)}_r$ are all weight-preserving poset-isomorphisms $[0_{n+1},\Psi^{(n)}_r(1_n)]\subseteq \mathcal P(n+1)$), that is $\omega_n(\pi)=\omega_{n+1}(\Psi^{(n)}_r(\pi))$
\end{definition}

\begin{example}
Indicator functions on all set partitions, and all non-crossing partitions are SI. Interval partitions are not SI. Indeed, there we have $$\omega_3^{\mathcal {I}}(\psi^{(2)}_2(1_2))= \omega_3^{\mathcal {I}}(\{\{1,3\},\{2\}\})=0\neq 1=\omega_2^{\mathcal {I}}(1_2).$$
The monotone weights are also not SI, as
$$\omega_3^{\mathcal {M}}(\psi^{(2)}_2(1_2))= \omega_3^{\mathcal {M}}(\{\{1,3\},\{2\}\})=1/2\neq 1=\omega_2^{\mathcal {M}}(1_2).$$

Weights for monotone cumulants are determined by the non-crossing partition's \emph{nesting tree factorial} $\omega^{\mathcal {M}}(\pi)=((\mathfrak{F}(\pi))!)^{-1}$ (see \cite{AHLV15}). Monotone weights are monic, multiplicative with respect to interval closures. They penalize nestings, generalizing factorials in the sense that $$\omega^{\mathcal {M}}_{2n}(\{\{1,2n\},\{2,2n-1\},\dots,\{n,n+1\}\})=1/n!.$$ A partition has non-zero monotone weight iff non-crossing, and the weight equals $1$ iff interval.

One may modify interval partitions and monotone weights, by considering \emph{almost-interval} partitions, to make them SI (Section 2.3). 
\end{example}

Now we reformulate Theorem \ref{mainthm} for this slightly more general setting. We will only include the proof for non-trivial $[\mathcal A,\mathcal B]$ and weights supported on $\mathcal {NC}$. The proof for the commutative case is similar.

\subsection{Main Theorem: Independence of Constants}

\begin{theorem}
Let $(\mathcal A,\mathcal B, \mathbb F)$ be an OVPS and let $\omega$ be an invertible SI-weight on $\mathcal {NC}$ from which we define the cumulants $(c^{\omega}_n)_n$. Then
$$c^{\omega}_{1_n}(a_1,\dots,a_n)=0,$$
whenever there is any $a_i\in \mathcal B$.
\end{theorem}

\begin{proof}
We proceed by induction on $n$. We will show that, for $n\geq 2$, the expression $$c_{1_n}(a_1,a_2,\dots,a_n)=0,$$ whenever there is some $a_i=b\in \mathcal B$.

For $n=1$, the SI property just means $\omega_1{\{1\}}=1$ and we have $c_1(b)=b$.

For $n=2$, SI means $1=\omega_1{\{1\}}=\omega_2(\{\{1\},\{2\}\})$ and thus $$\omega_2(\{\{1,2\}\})c^{\omega}_{1_2}(a,b)=\mathbb{F}(ab)-\mathbb{F}(a)\mathbb{F}(b)= \mathbb{F}(a)b-\mathbb{F}(a)b=0,$$ so $c^{\omega}_{1_2}(a,b)=0$ and similarly for $c^{\omega}_{1_2}(b,a)$ 

Suppose that for every $k\leq n$, if there is $i\in \{1,2,\dots, k\}$, with $a_i\in \mathcal{B}$, then $c^{\omega}_{1_k}(a_1,a_2,\dots ,a_k)=0.$ From the moment-cumulant formula over $\mathbb{F}(a_1\cdots a_{r-1} b \cdots  a_{n+1})$ and $\mathbb{F}(a_1\cdots (a_{r-1} b) \cdots  a_{n+1})$
we get
\begin{eqnarray}
     &&\sum_{\sigma\in \mathcal{P}(n+1)} \omega_{n+1}(\sigma)c^{\omega}_{\sigma} (a_1,\dots, b, \dots a_{n+1})  \\
     &=&\sum_{\pi\in \mathcal{P}(n)} \omega_n(\pi)c^{\omega}_{\pi} (a_1,\dots, a_{r-1} b, a_{r+1}, \dots a_{n+1}). 
\end{eqnarray}
There are three types of partitions for $\sigma\in \mathcal P (n+1)$ on the left-hand side of the equation:
\begin{enumerate}
    \item The full partition $1_{n+1}$ (to be solved to show that $\omega(1_{n+1})c^{\omega}_n(a)=0$).
    \item Partitions $\sigma\in [0_{n+1},\{\{1,\dots,r-1,r+1,\dots,n+1\},\{r\}\}]$, (i.e. $\{r\}$ is a singleton of $\sigma$).
    \item Those partitions in the complement. 
\end{enumerate}

For partitions in the last case, the block containing $b$ must have $2\leq k\leq n$ arguments, and thus the contribution vanishes by induction hypothesis.

The partitions from case ii) can be canceled-out by pairing them with the terms corresponding to $\pi=\Psi_r^{-1}(\sigma)\in \mathcal P(n)$ on the right-hand side, due to the SI condition. Thus for the remaining full-interval block, we have $\omega(1_{n+1})c^{\omega}_{1_{n+1}}(a)=0$. Since $\omega$ is invertible, $c^{\omega}_{1_{n+1}}(a)=0$.
\end{proof}

\subsection{New Examples of SI Families}

\begin{definition}
Consider the map $\mathrm {RS}:\mathcal P\to \mathcal P$ that removes all singletons from a partition.

For example: $\mathrm {RS}(\{\{1,4\},\{2\},\{3,6\},\{5\}\})=\{\{1,3\},\{2,4\}\}$.
\end{definition}

\begin{example}\label{ex:SIint}
As in \cite{Ora02} consider the set of almost interval partitions $\tilde {\mathcal I}(n):=\{\pi \in \mathcal {NC}(n): \mathrm{RS}(\pi)\in \mathcal I\}$. Then $\tilde {\mathcal I}:=\bigcup \tilde {\mathcal I}(n)$ is SI.
\end{example}

\begin{example}\label{ex:SImon}
The modified monotone weight $\tilde \omega^{\mathcal M}(\pi):=\omega^{\mathcal M}(\mathrm{RS}(\pi))$ is monic and SI. Similar to the usual monotone weights, we have that $\tilde \omega^{\mathcal M}$ is:

i) Supported on $\mathcal{NC}$

ii) Multiplicative w.r.t. interval closures.

iii) $1$ iff $\pi \in \tilde {\mathcal I}$.
\end{example}

\begin{example}\label{ex:SICI}
The cyclic-boolean, or cyclic-interval partitions, are the interval partitions when partitions are regarded on the circle (not on the semi-line). The only practical difference is that $n$ and $1$ are now considered \emph{consecutive}, which results in enlarging $\mathcal {I}$ to the lattice of cyclic-interval partitions $\mathcal {CI}$. For example, $\{\{1,5,6\},\{2,3\},\{4\}\}\in \mathcal {CI}(6)$, but not in $\mathcal I(6)$. 

$\mathcal {CI}$ is not-SI. However, if we ignore the singletons for the interval condition, as in Example \ref{ex:SIint}, the family of partitions $\widetilde{\mathcal {CI}}=\{\pi\in \mathcal {NC}: \mathrm{RS}(\pi)\in \mathcal{CI}\}$ becomes SI.
\end{example}

\begin{example}
Monotone-like weights can also be defined for partitions on the circle. One must only define a notion of interior/exterior for blocks, when non-crossing partitions are viewed in their circular representations. One possibility is to assign a nesting tree by declaring an outer block when there exist a (necessarily unique) block that contains the center of the circle. If no such block exists, the outer-blocks are the $k\geq 2$ blocks which can be seen from the center of the circle.

These weights are, again, not SI, but become so if we make $\tilde \omega^{\mathcal M}(\pi)=\omega^{\mathcal M}(\mathrm{RS}(\pi))$.
\end{example}

\begin{example}\label{ex:nonmonic}
Non-monic SI weights:

i) Singletons: $\omega(\pi)=1$ iff $\pi=0_k$ for some $k\geq 1$. 

ii) Weights supported on pairings with $\omega^{q}(\pi)=(q)^{\mathrm{cr}(\pi)}$, where $\mathrm{cr}(\pi)$ is the number of crossings of $\pi$, show-up for the moment-cumulant formulas of $q$-Gaussians (see \cite{BKS97}). Removing singletons does not affect the number of crossings, thus, the weight can be extended to one that does satisfy SI: $\tilde \omega^{q}(\pi)=(q)^{\mathrm{cr}(\mathrm{RS}(\pi))}$.
\end{example}

Let us briefly address the boolean probability theory with independent constants that is derived from Example \ref{ex:SIint}. 

This notion independence is called ``Fermi Boolean'' independence, first investigated by F. Oravecz in \cite{Ora02}. He introduced the almost interval partitions and provided a model for independence in terms of creation and annihilation operators. Oravecz also computed the non-centered Central limit theorem and a nice Poisson limit theorem (see also \cite{Ora04}, and Hasebe \cite{Has11}).

\section{Relation to Fermi-Boolean Probability}

The boolean independence with independent constants just mentioned is exactly the Fermi-boolean independence of F. Oravecz.

\begin{definition}
The Fermi-boolean cumulants $\tilde b_n$ are defined inductively by the moment-cumulant formula supported on the lattice of almost-interval partitions $\tilde {\mathcal I}$:
$$\mathbb{F}(a_1a_2\dots a_n)= \sum_{\pi\in \tilde {\mathcal{I}}(n)}\tilde  b_\pi(a_1,a_2,\dots, a_n).$$
\end{definition}

\begin{definition}
The algebras $\mathcal{A}_1,\mathcal{A}_2,\dots,\mathcal{A}_n$ are \emph{$\mathcal{B}$-Fermi-boolean independent} iff all cumulants $\tilde b_n$, when evaluated in mixed moments, vanish.
\end{definition}

What does this independence mean concretely for the moments of non-commutative random variables?. 

First, notice that mixed moments of \emph{centered} random variables \emph{are not affected} by the new combinatorics, and hence the factorization remains the same for these cases, namely $$\mathbb F(a_1a_1a_2a_2a_1a_1a_3a_3a_3a_2a_2a_1a_1)=\mathbb F(a_1^2)\mathbb F(a_2^2)\mathbb F(a_1^2)\mathbb F(a_3^3)\mathbb F(a_2^2)\mathbb F(a_1^2),$$
holds for centered, Fermi-boolean independent $(a_1,a_2,a_3)$.

As a consequence, the centered CLT has the same limiting law as in the usual boolean case (i.e. the centered symmetric Bernoulli distribution, see \cite{SW97}). Furthermore, from the centered case, the non-centered Fermi-boolean central limit follows directly using the standard proof for the classical and free case in terms of cumulants. The resulting limit distribution is just the shifted Bernoulli distribution.

For non-centered random variables, one really needs to solve a different moment-cumulant formula. As observed in \cite{Ora02}, the Fermi-boolean factorization is an instance of the so-called c-free independence: an extension of non-commutative probability where mixed moments are computed with respect to a pair of funcionals $(\tau,\psi)$, where the case $(\tau,\tau)$ recovers free independence w.r.t. $\tau$ and the case $(\tau,\delta_0)$, where $\delta_0$ is the Dirac measure at zero, yields the usual boolean independence (see \cite{BLS96}).

\subsection{Relation with c-free independence}

The case of Fermi-boolean independence with respect to $\tau$ coincides with the c-free independence with respect to the pair of functionals $(\tau,\psi)$, where $\psi(a^n)=(\tau(a))^n$, for all $n\geq 1$.

By considering pairs of conditional expectations $(\mathbb F_{\tau},\mathbb F_{\psi})$, one gets operator-valued c-free independence. The case where $\mathbb F_{\psi}(a^n)=(\mathbb F_{\tau}(a))^n$, $n\geq 1$ yields the factorization for  mixed moments of $\mathbb F_{\tau}$-Fermi-boolean independent operators.

Recall that in operator-valued free probability, the extremal case $\mathbb F=\mathrm{id}=\mathrm{id}_{\mathcal A}$, although trivial, is important because it guarantees that operators are free with respect to a conditional expectation onto a sufficiently large subalgebra $\mathcal B\subset \mathcal A$.

The combinatorics for the case $\mathbb F=\mathrm{id}$ collapse, in the sense that the moment-cumulant formulas are supported on the non-invertible lattice $\mathcal L(n):=\{0_n\}$, $n\geq 1$ from Example \ref{ex:nonmonic} i).

Consider now the extremal c-free cases $(\mathrm{id},\mathbb F)$ and $(\mathbb F,\mathrm{id})$. The moment-cumulant formulas in the first case are again supported on the non-invertible lattice $\mathcal L(n):=\{0_n\}$, $n\geq 1$. For the later case, the formula is supported on the almost interval partitions $\widetilde {\mathcal{I}}$.

Thus, the Fermi-boolean independence, from which we know a lot for its intersection with the usual boolean probability, seems more robustly related with free independence, as:
\begin{itemize}
    \item Cumulants appear more canonically in the framework of operator-valued c-free independence, and
    \item They now share the property of independent constants, in the framework of operator-valued independence
\end{itemize}

We expect that a combination of both extensions of non-commutative probability could lead to practical applications of boolean probability theory. 

For example, an approach to mixed moments of random unitary matrices (which are centered elements with nice boolean cumulants) and deterministic matrices (which are constants w.r.t the conditional expectation of entry-wise evaluation), does not seem unrealistic.

Before concluding this work, we compute, in the next section, the M\"obius functions of some of the new families of partitions from Section 3. M\"obius functions are important for expressing cumulants in terms of moments more directly, and are useful for deriving cumulant-to-cumulant formulas.

We obtain some interesting integer sequences and recursions.

\section{M\"obius Functions}

For any locally compact partially ordered set $P$, we may consider its incidence algebra, of complex-valued functions $f:P\times P\to \mathbb C$ on poset intervals, where multiplication is given by combinatorial convolutions (see \cite{Sta12}). 

Standard elements of the incidence algebra are the multiplicative unit $\delta$, the zeta function $\zeta$ and its inverse, known as the M\"obius function $\mu$. 
The values of the functions $\mu_{\mathcal {P}}$, $\mu_{\mathcal {NC}}$ and $\mu_{\mathcal {I}}$ over the full intervals $[0_n,1_n]$ are standard combinatorial sequences: $(-1)^{n-1}(n-1)!, (-1)^{n-1}C_{n-1}$ and $(-1)^{n-1}$ respectively, where $C_n$ are de Catalan numbers.

To warm-up, let us first describe the lattices of interval partitions and their M\"obius functions. 

It is convenient to observe that the interval partitions $\mathcal I(n)$ are in bijection with binary words of length $n-1$:  $B_{n-1}:=\{b=b_1b_2b_{n-1}, b_i\in \{0,1\}\}$. Indeed we may think of the $b_i$'s as $n-1$ buttons between $n$ consecutive numbers, where $b_r=1$, resp. $b_r=0$ indicate that the button is pressed/unpressed, and we associate it with a partition $\pi\in \mathcal I(n)$, where either $r\sim r+1$ or not (Figure 1, left).
\begin{center}
  \begin{tikzpicture}[x=1.8cm, y=1.6cm]
    \tikzstyle{dianode} = [minimum size=1.75mm,fill,circle,inner sep=0pt];


    \draw (0,1) node[dianode] (p111) {};
    \draw (p111)+(0,0.05) node[above] {\partition{1}{1}{1}};

    \draw (1,0) node[dianode] (p011) {};
    \draw (p011) node[above right] {\partition{0}{1}{1}};

    \draw (-1,0) node[dianode] (p110) {};
    \draw (p110) node[above left] {\partition{1}{1}{0}};

    \draw (0,-1) node[dianode] (p010) {};
    \draw (p010) node[below right] {\partition{0}{1}{0}};

    \draw (p111) -- (p011) -- (p010) -- (p110) -- (p111);

    \draw (0,0) node[dianode] (p101) {};
    \draw (p101) node[above right] {\partition{1}{0}{1}};

    \draw (1,-1) node[dianode] (p001) {};
    \draw (p001) node[below right] {\partition{0}{0}{1}};

    \draw (-1,-1) node[dianode] (p100) {};
    \draw (p100) node[below left] {\partition{1}{0}{0}};

    \draw (0,-2) node[dianode] (p000) {};
    \draw (p000) node[below] {\partition{0}{0}{0}};

    \draw (p101) -- (p001) -- (p000) -- (p100) -- (p101);

    \draw (p111) -- (p101);
    \draw (p011) -- (p001);
    \draw (p010) -- (p000);
    \draw (p110) -- (p100);


    \draw (3.5,0) node[dianode] (q) {};
    \draw (q)+(0,0.05) node[above] {\partitionc};

    \draw (2.5,-1) node[dianode] (q100) {};
    \draw (q100) node[left] {\partitionb{1}{0}{0}};

    \draw (4.5,-1) node[dianode] (q110) {};
    \draw (q110) node[right] {\partitiond};

    \draw (3.5,-2) node[dianode] (q000) {};
    \draw (q000) node[below] {\partitionb{0}{0}{0}};

    \draw (3.5,-1) node[dianode] (q010) {};
    \draw (q010) node[left] {\partitionb{0}{1}{0}};
 
    \draw (q) -- (q110) -- (q000) -- (q100) -- (q);
    \draw (q) -- (q000);

    \draw (q)+(0.1,0.23) node[right] {$\longleftrightarrow \left\{\!\!\begin{array}{l}\bullet\times\times, \\ \times\bullet\times, \\ \times\times\bullet\end{array}\!\!\right\}$};
  \end{tikzpicture}
  
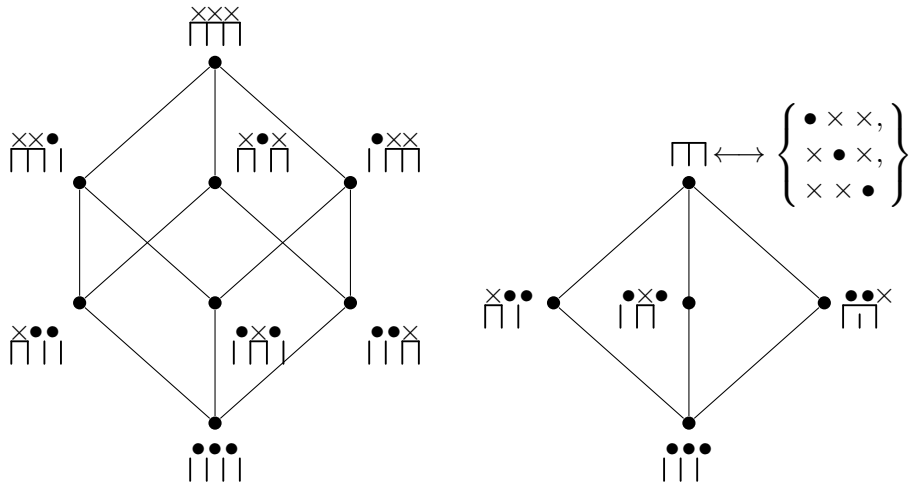
        Fig. 1: Hasse diagrams of $\mathcal I(4)$ (left) and $\mathcal {CI}(3)$ (right)
\end{center}

The Hasse Diagram is a hyper-cube (see Figure 1, left) and the M\"obius function is easily shown to be $(-1)^{n-1}$, since M\"obius functions are multiplicative w.r.t. cartesian products and $\mathcal I(n)\cong (\mathcal I(2))^{n-1}$. 

An alternative proof, which will be useful later, goes as follows:

Recall first Weisner's Lemma:  For every finite lattice $\mathcal{R}$ and every $\sigma\in \mathcal{R}\setminus\{0_{\mathcal{R}}\}$, $$\sum_{\pi\in \mathcal{R},\ \pi\vee \sigma=1} \mu_{\mathcal{R}}([0_{\mathcal{R}},\pi])=0.$$

We apply Weisner's Lemma for the full lattice (i.e. case $\sigma=1_n$). By induction, the sum in the left hand side is just the alternating sum of binomial coefficients
$$\sum_{\pi\in \mathcal I(n)} \mu_{\mathcal I}([0_n,\pi])= \mu_{\mathcal I}([0_n,1_n])+ \sum_{0\leq k\leq n-1}(-1)^k{n \choose k}=0.$$
Thus solving for $\mu_{\mathcal I}([0_n,1_n])$ gives $(-1)^{n-1}$.

We now will make analog arguments for cyclic-interval partitions.
\subsection{Cyclic-interval Partitions}

\begin{proposition}
The poset of cyclic-interval partitions $\mathcal{CI}(n)$ is isomorphic to the poset $\mathcal I(2)^{n}$, with the highest two levels of the Hasse diagram collapsed (see Fig. 1, right).
\end{proposition}
\begin{proof}
We use binary buttons again but this time there is an additional button between $n$ and $1$ (see Figure 1, right). Observe that this gives a bijection if we ignore word of $n+1$ ones (which now makes no sense), and we identify all words with a single zero with the full interval partition $1_n$. The assertion follows.
\end{proof}
\begin{corollary}
\begin{enumerate}
    \item All intervals $[\sigma, \pi]\subset \mathcal{CI}$ are either isomorphic to a smaller $\mathcal{CI}(k)$ (case $\pi=1_n$) or a smaller $\mathcal{I}(k)$ (case $\pi\neq 1_n$).
    \item The cardinality of $\mathcal{CI}(n)$ is $2^{n}-n$.
    \item The M\"obius functions for cyclic-interval partitions is $\mu_{\mathcal{CI}}([0_n,1_n])=(-1)^{n+1}(n-1)$. By (1), the M\"obius function can be obtained for any interval.
\end{enumerate}
\end{corollary}
\begin{proof}
For the first statement it is simple to see that the button bijection restricted to an interval $[\sigma,\pi]$ is an isomorphism to words of binary buttons, where some of the buttons have been already pressed (indicated by $\sigma$) and some buttons are forbidden (indicated by $\pi$).

The second statement is trivial. The third statement is proved using again Weisner's Lemma, applied to $\mathcal I(n+1)$ and $\mathcal {CI}(n)$. $$0=\sum_{\pi\in \mathcal {I}(n)} \mu_{\mathcal {I}}([0_{n+1},\pi])=\sum_{\pi\in \mathcal {CI}(n)} \mu_{\mathcal {CI}}([0_n,\pi])=\mu_{\mathcal {CI}}([0_n,1_n])+ \sum_{0\leq k\leq n-1}(-1)^k{n \choose k},$$ since the lower part of the lattice $\mathcal {CI}(n)$ is isomorphic to the lower part of $\mathcal I(n+1)$. Thus we may cancel most summands. Solving for $\mu_{\mathcal {CI}}([0_n,1_n])$ yields the result $$\mu_{\mathcal {CI}}([0_n,1_n])=\mu_{\mathcal {I}}([0_{n+1},1_{n+1}])+(-1)^{n-1}{n \choose n-1}=(-1)^{n}+n(-1)^{n-1}.$$
\end{proof}
Now, we move our attention to almost-interval partitions

\subsection{Almost-interval Partitions}

\begin{proposition}
The number of elements in the poset of almost-interval partitions $|\widetilde{\mathcal{I}}(n)|=F_{2n-1}$ are the odd Fibonacci numbers $(1,2,5,13,34,89,\dots)$ .
\end{proposition}
\begin{proof}
For $n=1,2$ the assertion is trivial. For $n\geq 3$ we use induction: the set $\widetilde{\mathcal {I}}(n+1)$ can be decomposed into disjoint subsets $I_1,\dots I_{n+1}$, where:

i). $I_{1}$ is the subset that collects all partitions where $1$ is a singleton: $\{\sigma \in \widetilde{\mathcal {I}}(n+1), \{\{1\}\}\in \sigma\}$. By SI, these are in bijection with the set $\widetilde{\mathcal {I}}(n)$.

ii). $I_2$ is the subset $\sigma \in \widetilde{\mathcal {I}}(n+1),$ where $1\sim 2$. These are in bijection with the set $\widetilde{\mathcal {I}}(n)$, by collapsing $1$ and $2$. 

iii). Similarly, for each $2\leq r\leq n+1$, $I_r$ is the subset that collects all partitions for which the first neighbor of $1$ to the right is $r$. Observe that this forces $2,3,\dots,r-1$ to be singletons in $\sigma$. By collapsing $1$ with $r$ and the SI condition used $r-2$ times, these are in bijection with $\widetilde{\mathcal {I}}(n+2-r)$, $r=2,3,\dots, n-1$.

Thus the set $\widetilde{\mathcal {I}}(n+1)$ is partitioned into $n+1$ subsets with cardinalities $|\widetilde{\mathcal {I}}(n)|,|\widetilde{\mathcal {I}}(n)|,|\widetilde{\mathcal {I}}(n-1)|,|\widetilde{\mathcal {I}}(n-2)|,\dots ,|\widetilde{\mathcal {I}}(1)|$. 

By induction hypothesis (using $F_1=F_2$) and an $n$-fold chain-application of the Fibonacci recursion $F_n+F_{n-1}=F_{n+1}$, we get $$|\widetilde{\mathcal {I}}(n+1)|=2F_{2n-1}+\cdots + F_{3}+ F_{1}=2F_{2n-1}+\cdots + F_{3}+F_{2}=F_{2n+1}.$$
\end{proof}
To tackle the problem of computing the values of the M\"obius function, we invoke again Weisner's Lemma.
\begin{theorem} \label{thm:moeb}
The following formula holds for the M\"obius functions of almost-interval partitions:
$$\mu_{\widetilde{\mathcal{I}}}([0_n,1_n])=-2\mu_{\widetilde{\mathcal{I}}}([0_{n-1},1_{n-1}]),$$ for $n\geq 3$, and hence $(\mu_{\widetilde{\mathcal{I}}}([0_n,1_n]))_n$ is just the sequence $(1,-1,2,-4,8,\dots)$
\end{theorem}

\begin{proof}
Let $n\geq 3$ and apply Weisner's Lemma for $\sigma=\{\{1,2\},\{3\},\dots, \{n\}\}\in \widetilde{\mathcal{I}}(n)$. Then $\pi\in \widetilde{\mathcal{I}}(n)$, $\pi \vee \sigma = 1_n$ if and only if 
$$\pi= 1_{\widetilde{\mathcal{I}}(n)},\quad \ \pi=\{\{1\},\{2,3,\dots,n\}\}, {\rm or}\quad  \pi=\{\{2\},\{1,3,\dots, n\}\}.$$ Since $\widetilde{\mathcal{I}}$ is SI, it follows that $$-\mu_{\widetilde{\mathcal{I}}}([0_n,1_n])=2\mu_{\widetilde{\mathcal{I}}}([0_{n-1},1_{n-1}]).$$
\end{proof} 

\begin{remark}
All intervals $[\sigma,\pi]$, $\sigma\leq \pi \in \widetilde{\mathcal{I}}(n)$ can be shown to be isomorphic to the Cartesian product of smaller $\widetilde{\mathcal{I}}(r)$'s and ${\mathcal{I}}(k)$'s. Thus, the M\"obius function can be computed on arbitrary intervals and the values are always of the form $\{\pm 2^k\}$, $k\geq 0$.
\end{remark}

\end{document}